\documentclass[11pt, onecolumn]{article}
\usepackage[height=22cm,width=17.5cm]{geometry}
 \usepackage{pstricks}
\usepackage{mathrsfs}
\usepackage{enumerate}
\usepackage{bbm}
\usepackage{tikz}
\usepackage{url}
\usepackage{hyperref}
\usetikzlibrary{arrows}
\usepackage{relsize}
\usepackage{pgfplots}
\usepackage{epsfig}
\usepackage[cmex10]{amsmath}
\usepackage{amssymb}
\usepackage{amsthm}
\usepackage{amsfonts}
\usepackage{bm}
\usepackage{cite}
\usepackage{algpseudocode}
\usepackage{graphicx}
\usepackage{lipsum}
\usepackage{tikz} 
 \usepackage{scrextend}

\newtheorem{theorem}{Theorem}

\newtheorem{lemma}{Lemma}
\newtheorem{remark}{Remark}
\newtheorem{definition}{Definition}
\def\pbb {\mathbb{P}}

\begin{document}
 
\date{}

\title{ 
On the Identifiability of Finite Mixtures of Finite Product Measures
}
\author{
Behrooz~Tahmasebi, Seyed~Abolfazl~Motahari, and    Mohammad~Ali~Maddah-Ali
\thanks{
Behrooz Tahmasebi is with the Department of Electrical Engineering, Sharif University of Technology, Tehran, Iran. (email: behrooz.tahmasebi@ee.sharif.edu).
  Seyed Abolfazl  Motahari is with the Department of Computer Engineering, Sharif University of Technology, Tehran, Iran.  (email: motahari@sharif.edu).
Mohammad Ali Maddah-Ali is with  Nokia Bell Labs, Holmdel, New Jersey, USA.  (email: mohammad.maddahali@nokia-bell-labs.com).
}
\thanks{
This paper has been presented in part at IEEE International Symposium on Information Theory (ISIT) 2018\cite{early}.
}
   }
   
\renewcommand\footnotemark{}
 
\maketitle

\begin{abstract} 
The problem of identifiability of finite mixtures of  finite product measures is studied. 
A mixture model with $K$ mixture components and $L$ observed variables is considered, where each variable takes  its value in a finite set with cardinality $M$.
The variables are independent in each mixture component.
 The identifiability of a mixture model means the possibility of attaining the mixture components parameters by observing its mixture distribution.
  In this paper, 
we investigate fundamental relations between the  identifiability of  mixture models and the separability of their observed variables by introducing two types of separability: strongly and weakly separable variables. 
Roughly speaking, a  variable is said to be separable, if and only if it has some differences among its probability  distributions in  different mixture components.
We prove that  mixture models are identifiable if the number of strongly separable variables is greater than or equal to $2K-1$, independent form    $M$. This fundamental threshold is shown to be tight, where a family of non-identifiable mixture models with less than $2K-1$ strongly separable variables is provided. 
 We also show that  mixture models are identifiable if they have at least $2K$ weakly  separable variables. 
 To prove these theorems, we introduce a particular polynomial, called characteristic polynomial,  which  translates the identifiability conditions to identity of polynomials and allows us to construct an inductive proof.

\end{abstract}


\section{Introduction}

Statistical analysis of samples from mixtures of subpopulations is usually carried out by assuming that the underlying probability law is governed by a mixture model. There exists a large number  of applications 
where mixture models are central part of the data analyses \cite{book1,book2,pritch,betamix}.  
An important  one is  in population genetics, where the  mixed population  datasets are modeled by finite mixture models \cite{pritch}. 
A finite mixture model can be represented by 
\begin{align*}
\pbb_{\theta} = \sum_{k=1}^K w_k  \pbb_{k},
\end{align*}
where  $\pbb_k$ and $w_k$ are, respectively, the probability law governing  and the relative population size of the $k^\text{th}$ subpopulation. $\theta$ encapsulates all the latent parameters used to specify the model including $w_k$s. 

In a common setting, a mixture model with latent parameters is assumed to be the law governing a given dataset. Then, latent parameters are estimated by various methods such as maximum likelihood estimation. However, before starting the parameter estimation one needs to answer a very important and fundamental question: is the model identifiable? Identifiability means that there is a one to one map between the latent parameter $\theta$ and the probability law $\pbb_{\theta}$ (up to permutations of the subpopulations which will be discussed later.) \cite{15,16,27,28}.

In this paper, we consider finite mixture models with $L$ observed variables  which have the conditional independence property.
This means that  $L$  variables are distributed in each mixture component according to a product probability measure, i.e., 
$\pbb_k =  \bigotimes_{\ell \in [L]} \pbb_{k\ell}$
 for any $k\in [K]$ (see Section 2 for more explanations).
In addition, we assume that the observed variables are discrete random variables and take values in a finite   set with cardinality $M$.
This finite set is also refereed  to as the state space in the literature.  
This class of mixture models are called finite mixtures of finite product measures.
In this paper, 
by mixture models we mean this class of distributions.

Studying identifiability of statistical models has a long history \cite{15,16,27,28}.  For  finite mixtures of finite product measures the problem has been studied in various articles (see for instance \cite{ide,ide2}).  
It is well known that there are  some Bernoulli Mixture Models (BMMs) which are not identifiable \cite{22,23}. 
We note that, as mentioned in \cite{ide2}, only counting the dimensionality of the set of the latent parameters and then  comparing it to the dimension of the mixture distribution is not sufficient to establish an identifiability argument. 
A novel method has been used in \cite{ide},  based on the seminal result of Kruskal \cite{29,30},  where the authors show that the parameters    of a finite mixture of finite product measures are generically identifiable if $L  \ge 2 \lceil   \log_{M}(K)\rceil + 1$.
This means that if $L  \ge 2 \lceil   \log_{M}(K)\rceil + 1$, then the measure of the parameters that make the problem non-identifiable is zero.

In many applications, such as population genetics \cite{pritch}, after modeling the drawn dataset by mixture models, the next step is  to cluster the given dataset into subpopulations (called population stratification in the population genetics literature).
The key observation is that many observed variables are not useful for  clustering, which means their   probability distributions (pmfs)   are  (exactly)  same in the different mixture components. 
This means the set of  possible parameters for this type of mixture models  has a zero measure.
On the other side, a small fraction of the observed variables are useful for clustering, which means that they have enough separation among their pmfs to make the reliable clustering possible \cite{najaf}.
 A fundamental (and also theoretical) question  is that what is  the number of required separable variables for a mixture model  to   be identifiable? 
 We note that, for this case,  we can not use the result of generic identifiability because the parameters are in a zero measure set. 
  In this paper
we establish  connections  between identifiability of the parameters of a mixture model and the separability of its variables.   
 In this way, two notions of separability are introduced:  strongly and weakly separable variables, which will be defined in the next two paragraphs. 
 We note that our studies in this paper can be taught as a  worst-case  analysis of the identifiability of the finite mixtures of finite product measures.

In a \textit{strongly separable} variable,  the probability distributions (pmfs) of the variable in the $K$ mixture components are strictly different from each other, for  each   element of the state space. This is a natural definition for  \textit{good} variables, as the  strongly separable variables are essentially useful for  clustering the datasets drawn by mixture models.
In \cite{najaf},  a mathematical analysis of the problem of reliable clustering of datasets drawn by mixture models based on the separability of mixture components is presented. 
However, the identifiability of mixture models based on the separability conditions is not explored in \cite{najaf}. 
We  establish a condition on the number of strongly separable  variables that results the  identifiability of the parameters of the corresponding mixture model. 
In particular,   we find a  sharp threshold that the identifiability is  guaranteed, if   the number of strongly  separable  variables is greater than or equal to it. 
The threshold is $2K-1$, where $K$ is the number of mixture components. We note that this threshold is independent of $M$, the size of the state space of  the variables, which is in contrast to the result of the generic identifiability. 
 Also we show that this threshold is tight by introducing a family of non-identifiable mixture models with less than $2K-1$ strongly separable variables.

We note that  for a strongly separable variable, the number of conditions which are needed to be satisfied scales by $M$.
To study the effect of this scaling on identifiability we introduce \textit{weakly separable} variables.
In a  \textit{weakly separable} variable, the probability distributions (pmfs) of the variable in the $K$ mixture components are strictly different from each other, for at least one (not necessarily all) state space element.
This means that for any weakly separable variable, the number of required conditions does not scale with  $M$. 
Observe that  by  definitions  any strongly separable variable is also weakly separable. 
Therefore, naturally we expect that the number of required weakly separable variables for the identifiability would be very larger than $2K-1$, due to the order-wise fewer number of conditions.
However, we prove that if a mixture model has at least $2K$ weakly separable variables, then it is identifiable. 
This means that the penalty of considering   weakly separable variables instead of strongly separable variables is at most one.

We notice that the threshold of $2K-1$ is also observed in the    identifiability of the other problems, like the Latent Block Models (LBMs)  \cite{est}, the  binomial mixture models \cite{17},  the mixture models from grouped samples \cite{grouped,grouped2}  and also the topic modeling problem \cite{topic}. 
However, the model of this paper is essentially  different from them and their proofs ideas are also very different from what we develop in this paper. 
For the case of binomial mixture models,  the result of $2K-1$ can be  followed as a special case of our results, when the   mixture components have i.i.d. variables.

To prove the sufficiencies in both $2K-1$ strongly and $2K$ weakly  separable variables, we introduce a multi-variable polynomial, called the characteristic polynomial representing a mixture model. Then, we  prove that the identifiability of the mixture model is equivalent to the identifiability of its characteristic polynomial in the polynomials space. This allows us to exploit properties of  polynomials to proof the identifiability argument. For the converse, we introduce  a family of non-identifiable mixture models that have less than $2K-1$ strongly  separable variables for arbitrary $K,L$ and $M$.

The rest of the paper is organized as follows. In Section \ref{Problem Statement},  we  define the problem. In Section 
\ref{Main Results},  we  describe the main results of  the paper. Finally, Sections 
\ref{Proof of Theorem 1} and 
\ref{Proof of Theorem 2}  contain  the proofs of the main theorems of the paper.

\textbf{Notation.} In this paper  all vectors are columnar and they are denoted by bold letters, like $\bold{x}$.  
 For any positive integer $N$, we define $[N]:=\{1,2,\ldots,N\}$.  The transpose operator is denoted by $(.)^T$. 
 The  polynomial identity is denoted by $\equiv$ which is used when   two polynomials  have  the same coefficients.  
We use the notation  $\bold{x}\otimes \bold{y}$ for the 
Kronecker product of two vectors  
$\bold{x}$ and $\bold{y}$. 
For any vector $\bold{x} = (x_1,x_2,\ldots,x_n)^T \in \mathbb{R}^n$, we define $\| \bold{x} \|_{0} := \sum_{i=1}^n \mathbbm{1}\{x_i \neq 0\}$. 
Let us define 
$$
\Delta_M =  \Big \{  \bold{x} = (x_1,x_2,\ldots,x_M)^T   \in \mathbb{R}^M  \Big | \sum_{m=1}^M x_m = 1 ,    \forall m \in [M]  :  x_m  > 0 \Big  \}.
$$
Also $\overline{\Delta}_M$ denotes the closure  of $\Delta_M$.

%%%%%%%%%%%%%%%%%%%%%%%%%%%
\section{Problem Statement}
\label{Problem Statement}

We consider a  finite mixture model with  $K$ mixture components and $L$ observed variables where each observed variable takes its value in a finite set (finite state space) with cardinality $M$.
For any mixture component, we have a generative model that  the observed variables attain their realizations based on that. 
In this model, the  ${\ell}^{\text{th}}$ variable   in the $k^{\text{th}}$ mixture component  is generated from a  finite  probability measure denoted by  $\bold{f}_{ k\ell}\in \overline{\Delta}_{M}$ (we call $\bold{f}_{   k\ell}$ the  frequency vector of the $\ell^\text{th}$ variable in the $k^{\text{th}}$ mixture component). 
 We assume the conditional independence structure in the mixture model, which means that the observed variables are independent in each mixture component. 
  Let us denote the frequency of the $\ell^\text{th}$ variable, in the $k^\text{th}$ mixture component,  for any $m\in[M]$, by  ${f}_{ k \ell m}$,  as it is also denoted by $\bold{f}_{   k\ell}(m)$.
We denote the collection of the frequencies specifically for  the $k^\text{th}$  mixture component by  a matrix   $F_k=  ( \bold{f}_{   k1}|\bold{f}_{   k2}|\ldots|\bold{f}_{   kL})^T  = (f_{k \ell m})_{L \times M} \in [0,1]^{L \times M}$.
 Let $F_{1:K}:=(F_1,F_2\ldots,F_K)$  denotes the collection of  the frequencies in  the mixture model in a  specific order.

In our model, 
each data instance is generated  from one of the mixture components according to a sampling distribution  with  pmf $\bold{w}=(w_1,w_2,\ldots,w_K)^T\in \Delta_K$. 
 Let us denote the latent parameters of the mixture model by $\theta  = (F_{1:K};\bold{w})$.
 We  also denote  
 the set of all possible latent parameters of the mixture models (as defined above) by $\Theta_{K,L,M}$. 
 Also let  $\overline{\Theta}_{K,L,M}$ denotes  the closure of $\Theta_{K,L,M}$.
Here,   $\theta \in \Theta_{K,L,M} \subseteq  \overline{\Theta}_{K,L,M}$\footnote{
The main difference of $ \Theta_{K,L,M}$ and $ \overline{\Theta}_{K,L,M}$ is that in the first, the probability of sampling from each mixture component is positive,  but in the second, it may be zero.}.

 In this paper, our objective is to attain the latent  parameters $\theta = (F_{1:K};\bold{w})$, using the mixture distribution. The  formal definition of  the  mixture distribution  is as follows.
\begin{definition}
The mixture distribution  $\bold{f}\in \overline{\Delta}_{M^L}$ of a mixture model with latent parameters
$\theta = (F_{1:K};\bold{w})$
 is defined as 
\begin{align*}
 \bold{f}:= \sum_{k=1}^K w_k  \times (\bold{f}_{k1}\otimes \bold{f}_{k2} \otimes  \cdots \otimes   \bold{f}_{kL}).
\end{align*}
In other words,   for any $(m_1,m_2,\ldots,m_{L}) \in [M]^{L}$, we  define  
\begin{align*}
\bold{f}{(m_1,m_2,\ldots,m_{L})} := \sum_{k=1}^K w_k \prod_{\ell=1}^{L} \bold{f}_{   k\ell}(m_{\ell}).
\end{align*}
Here  $\bold{f}$ is the vector containing all   $\bold{f}{(m_1,m_2,\ldots,m_{L})}$ for any $(m_1,m_2,\ldots,m_{L}) \in [M]^{L}$ in a  specific order.
\end{definition}

In the following definitions, we  define the notion of the  identifiability of a mixture model.

\begin{definition}
For  any  $(F_{1:K};\bold{w})\in \overline{\Theta}_{K,L,M}$ and any permutation $\pi$ on $[K]$,  we define $(F^{\pi}_{1:K};\bold{w}^{\pi})\in \overline{\Theta}_{K,L,M}$ as follows. 
$$F^{\pi}_{1:K}:=(F_{\pi(1)},F_{\pi(2)},\ldots,F_{\pi(K)}),$$
and
$$\bold{w}^{\pi}:=(w_{\pi(1)},w_{\pi(2)},\ldots,w_{\pi(K)})^T.$$
\end{definition}
\begin{definition}
For  any $(F_{1:K};\bold{w}),(G_{1:K};\bold{z})\in \overline{\Theta}_{K,L,M}$, if there exists a permutation $\pi$ on $[K]$ such that $(F_{1:K};\bold{w})=(G_{1:K}^{\pi};\bold{z}^{\pi})$, then we write $(F_{1:K};\bold{w}) \approx (G_{1:K};\bold{z})$.
\end{definition}

\begin{definition}
A mixture model with  latent parameters $\theta = (F_{1:K};\bold{w})\in {\Theta}_{K,L,M}$ and mixture distribution $\bold{f}$ is said to be identifiable, iff  for any $(G_{1:K};\bold{z})\in \overline{\Theta}_{K,L,M}$ with mixture distribution $\bold{g}$, such that $\bold{f} = \bold{g}$ (they have the same mixture distributions),  we have  $(F_{1:K};\bold{w}) \approx (G_{1:K};\bold{z})$. 
In other words, a mixture model   is identifiable, iff its latent parameters  $\theta = (F_{1:K};\bold{w})$ can be identified from its mixture distribution  $\bold{f}$, up to a label swapping.
\end{definition}

For the identifiability of a mixture model it is essential that  
 the mixture components separate in some measures.
In this paper, first we define the notion of \textit{ strongly separable } variables as the measure of  difference among  mixture components. 
In our definition, the $\ell^{\text{th}}$ variable is said  to be   strongly separable,  if and only if  the  frequencies of it  are different for any distinct mixture components and also for any state space element.  
In what follows,  we mathematically define the notion of the strongly separable variables of a mixture model.
\begin{definition}
Consider a mixture model  with  latent parameters $\theta = (F_{1:K};\bold{w})\in \Theta_{K,L,M}$. 
In this mixture model,   the $\ell^{\text{th}}$ variable is said to be  strongly separable,  iff for any $m \in[M]$ and any distinct $k_1,k_2\in[K]$, we have ${f}_{ k \ell  m}\neq  {f}_{ k'\ell  m}$. In other words, the $\ell^{\text{th}}$ variable is strongly separable, iff  for any $k_1,k_2\in[K]$ we have  $\| \bold{f}_{k\ell} -\bold{f}_{k'\ell} \|_0 = M$. 
  Also the number of strongly separable variables of a mixture model  with latent parameters $\theta = (F_{1:K};\bold{w})$ is denoted by $\mathcal{L}_s{(F_{1:K};\bold{w})}$.
\end{definition}

\begin{remark}\normalfont
We note that  strongly separable variables are global.
In other words, in a strongly separable variable, any pair of  mixture components have different frequencies. 
\end{remark}

We note that the number of conditions that must be satisfied for the strong separability of a variable scales with
the state space size $M$.
For  studying  the effect of  this scaling on identifiability,  we  define \textit{ weakly separable } variables in this paper. 
The definition of the weakly separable variables is as follows.

\begin{definition}
Consider a mixture model  with  latent parameters $\theta = (F_{1:K};\bold{w})\in \Theta_{K,L,M}$. In this mixture model,   the $\ell^{\text{th}}$ variable is said to be  weakly separable,  iff   there is an  $m \in[M]$, such that for any distinct $k_1,k_2\in[K]$  we have ${f}_{ k \ell  m}\neq  {f}_{ k'\ell  m}$
\footnote{
  Note that in this case, for any distinct $k_1,k_2\in[K]$, we have 
  $\| \bold{f}_{k\ell} -\bold{f}_{k'\ell} \|_0 \ge 1$. 
  }. 
  Also the number of weakly separable variables of a mixture model  with latent parameters $\theta = (F_{1:K};\bold{w})$ is denoted by $\mathcal{L}_w{(F_{1:K};\bold{w})}$  . 
\end{definition}

\begin{remark}\normalfont
Observe that any strongly separable variable is also   weakly separable. 
This means that we have the inequality $\mathcal{L}_w{(F_{1:K};\bold{w})} \ge  \mathcal{L}_s{(F_{1:K};\bold{w})}$ for any latent parameters $(F_{1:K};\bold{w})\in \Theta_{K,L,M}$. 
Specially  for the case of binary state spaces $(M=2)$  the definitions of the weakly separable variables and the strongly separable variables are  same, resulting   $\mathcal{L}_w{(F_{1:K};\bold{w})} =  \mathcal{L}_s{(F_{1:K};\bold{w})}$ for any  $(F_{1:K};\bold{w})\in \Theta_{K,L,2}$. 
\end{remark}

%
%\begin{remark}\normalfont
%In this definition, the number of  required inequalities for the weak  separability of a variable does not scale by $M$. Also, notice that the weakly  separable variables are global, which means that in a weakly separable variable, any pair of  mixture components have different frequencies, similar to the strongly separable variables. 
%\end{remark}

In the next section, we analyze the problem of identifiability of the  mixture models and its relation to the number of separable variables (strong or weak). 
In particular,  we show that there is a sharp threshold on the number of separable variables (strongly or weak), that implies the identifiability of the corresponded mixture model.

\section{Main Results}
\label{Main Results}
The main result of this paper for the strongly separable variables is summarized in  the following theorem. 

\begin{theorem}\label{Theorem1} (Necessary and sufficient condition for the  identifiability based on the strongly separable variables) 
Any mixture model with latent variables $\theta = (F_{1:K};\bold{w}) \in \Theta_{K,L,M}$, such that $\mathcal{L}_s{(F_{1:K};\bold{w})} \ge 2K-1$,  is identifiable. Conversely, for any $\mathcal{L}_s{(F_{1:K};\bold{w})} \le 2K-2$, there is a mixture model with latent variables $\theta = (F_{1:K};\bold{w}) \in \Theta_{K,L,M}$, which has  $\mathcal{L}_s{(F_{1:K};\bold{w})}$ strongly separable variables  and is not identifiable.

\end{theorem}

\begin{remark}  \normalfont  
 The theorem states that if $\mathcal{L}_s{(F_{1:K};\bold{w})}\ge 2K-1$ then the  identifiability  of  the parameters from the mixture distribution is guaranteed.  On the other hand,  if $\mathcal{L}_s{(F_{1:K};\bold{w})}  \le 2K-2$ there is no guarantee that the problem is identifiable. Hence, in the  worst-case analysis, the identifiability is possible if and only if we have at least $2K-1$ strongly separable variables. 
\end{remark}

\begin{remark}  \normalfont  
Our threshold depends only on the number of mixture components $K$, and it is independent of $M$.  Naturally,  we expect that our threshold varies with the size of the state space,  similar to the generic identifiability result, which is $2 \lceil   \log_{M}(K)\rceil + 1$. 
Also, the threshold is $\Theta(K)$, while in the generic identifiability, the threshold  is $\Theta(\log_M (K))$.  This means that by relying on strongly separable variables  order-wise more variables are required in comparison with generic identifiability.
\end{remark}

\begin{remark}\normalfont
For the proof of the theorem, first we introduce a multi-variable polynomial which is made up by the parameters of the problem. 
We show that the identifiability of a mixture model   follows by the identifiability 
of its characteristic polynomial in
 the class of polynomials. 
Then,  by exploiting  the properties of the  polynomials, we prove the sufficiency part of the theorem. For the necessary  part, we introduce a family of non-identifiable mixture models that shows the necessity of the condition in the worst-case regime.
\end{remark}

The result  for  the   weakly separable variables is also provided in the next theorem.

\begin{theorem}  (Sufficient condition for the  identifiability based on the weakly separable variables) 
Any mixture model with latent variables $\theta = (F_{1:K};\bold{w}) \in \Theta_{K,L,M}$  such that $\mathcal{L}_w{(F_{1:K};\bold{w})} \ge 2K$  is identifiable.
\end{theorem}

\begin{remark}  \normalfont  
 The theorem states that although the weakly separable variables satisfy order-wise less conditions, but to use them, it suffices to have just  one extra variable,  in comparison with the strong separability measure. 
 \end{remark}

 \begin{remark}  \normalfont  
It is easy to see that in a more general model,  if the observed  variables have state spaces with (possibly) different sizes  denoted by $\big \{ M_\ell \big \}_{\ell \in [L]}$, then the notion of weakly separable variables can be defined for them similarly and also the sufficiency  of  $2K$ weakly separable variables for the identifiability of them holds.  For this matter, it just suffices to set $M = \max_{\ell \in [L]} M_{\ell}$ and then consider each observed variable as an instance with state apace of cardinality $M$, by setting $M-M_{\ell}$ frequencies to be zero in the $\ell^{\text{th}}$ variable.
 \end{remark}

The proofs of the theorems are available in the next two sections.

\section{Proof of Theorem 1}
 \label{Proof of Theorem 1}
The proof consists of three steps. First we  prove the sufficiency part of theorem for  the binary state space,  i.e.,  the case  $M=2$.   
Then we extend the  proof  for non-binary state spaces. The necessity part of theorem is also proved via  providing  a class of non-identifiable mixture models. This three steps of the proof are available in the next three subsections.

\subsection{Proof of the sufficiency part of Theorem 1 for $M=2$}
In this part, we use the notation $f_{k\ell}$ instead of $f_{k\ell1}$. Note that we  have $f_{k\ell 1}+ f_{k\ell 2} = 1$ in the binary case. Also, we use  ${\Theta}_{K,L}$ and $\overline{\Theta}_{K,L}$ instead of ${\Theta}_{K,L,2}$ and $\overline{\Theta}_{K,L,2}$, respectively.
First we need to define the characteristic polynomial of the latent parameters.

\begin{definition}
The characteristic polynomial of latent parameters $(F_{1:K};\bold{w})\in \overline{\Theta}_{K,L}$  is an $L$-variable polynomial that is defined as follows.
\begin{align*}
C_{(F_{1:K};\bold{w})}(x_1,x_2,\ldots,x_L):= \sum_{k=1}^K w_k\prod_{\ell=1}^L(x_{\ell}-f_{k\ell }).
\end{align*}
We  also denote the characteristic polynomials by $C_{(F_{1:K};\bold{w})}(x_{1:L})$ in a brief way.

\end{definition}
 The characteristic polynomial of latent parameters $(F_{1:K};\bold{w})\in {\Theta}_{K,L}$ has an important role in our proofs. In particular, in the next lemma, we prove that the identifiability of the characteristic polynomial of  a mixture model  implies  the the identifiability of the corresponding mixture model.

 \begin{lemma}\label{lemma1}
For any $(F_{1:K};\bold{w})\in {\Theta}_{K,L}$,  the  following propositions are equivalent. 
 \begin{enumerate} [(i)]
\item A mixture model with the latent parameters $(F_{1:K};\bold{w})\in {\Theta}_{K,L}$
is identifiable.
\item   
For any  $(G_{1:K};\bold{z})\in \overline{\Theta}_{K,L}$ satisfying the polynomial identity $C_{(G_{1:K},\bold{z})}(x_{1:L}) \equiv C_{(F_{1:K};\bold{w})}(x_{1:L})$,  we have  $(F_{1:K};\bold{w})\approx (G_{1:K};\bold{z})$.
\end{enumerate}

 \end{lemma}

 \begin{proof}
 See appendix A.
 \end{proof}
 
 \begin{remark} \normalfont
Lemma \ref{lemma1} shows that the identifiability of mixture models of products of Bernoulli measures is equivalent to the identifiability of a class of multi-variable polynomials. This connection makes it possible to prove an identifiability result in the class of multi-variable polynomials and use it to prove the identifiability of mixture models.
 \end{remark}

Now,  we state the following theorem which concludes the proof of the sufficiency part of Theorem \ref{Theorem1} for binary case.

\begin{theorem}\label{Theorem3}
Let  $(F_{1:K};\bold{w})\in \Theta_{K,L}$ with  $\mathcal{L}_s(F_{1:K};\bold{w})\ge 2K-1$.  
Then, for any  $(G_{1:K};\bold{z})\in \overline{\Theta}_{K,L}$ the identity  $C_{(G_{1:K};\bold{z})}(x_{1:L}) \equiv C_{(F_{1:K};\bold{w})}(x_{1:L})$  implies $(F_{1:K};\bold{w}) \approx (G_{1:K};\bold{z})$.
\end{theorem}

 \begin{remark} \normalfont
 Observe that the sufficiency part of Theorem \ref{Theorem1} for binary state spaces follows by  Theorem \ref{Theorem3} and Lemma \ref{lemma1}. Also, note that Theorem \ref{Theorem3} establishes an identifiability argument for a class of multi-variable polynomials, named by characteristic polynomials.
 \end{remark}

\begin{proof}

To prove Theorem \ref{Theorem3}, we establish an stronger argument. We relax the condition 
$\bold{w} \in \Delta_K$ in the definition of $\Theta_{K,L}$ to  $\| \bold{w} \|_0 = K$  and form a new set of latent parameters, denoted by $\Theta^*_{K,L}$. Note that we have $\Theta_{K,L}\subseteq  \Theta^*_{K,L}$. Also 
$ \overline{\Theta^*}_{K,L}$ denotes the closure of $ \Theta^*_{K,L}$. 
Now we prove the statement of the theorem, for the cases  that  $(F_{1:K};\bold{w})\in \Theta^*_{K,L}$ and $(G_{1:K};\bold{z})\in \overline{\Theta^*}_{K,L}$, which is  stronger than  the theorem\footnote{
Note that all of the prior  definitions for $(F_{1:K};\bold{w})\in \Theta_{K,L}$ naturally extend to the elements of $\Theta^*_{K,L}$.   
}. The reason that we use this modification is that this allows us to establish an  inductive proof.

The proof is based on an    induction on $K$. The case  $K=1$ is trivial. Assume that the theorem is proved for any  $K < \tilde{K}$. We will prove   the theorem   for $K=\tilde{K}$. Assume that for some $(F_{1:\tilde{K}};\bold{w}) \in {\Theta^*}_{\tilde{K},L}$ and $ (G_{1:\tilde{K}};\bold{z})\in \overline{\Theta^*}_{\tilde{K},L}$, we have the identity $C_{(F_{1:\tilde{K}};\bold{w})}(x_{1:L}) \equiv C_{(G_{1:\tilde{K}};\bold{z})}(x_{1:L})$ and also $\mathcal{L}_s{(F_{1:\tilde{K}};\bold{w})}\ge 2\tilde{K}-1$. We will show that
 $(F_{1:\tilde{K}};\bold{w}) \approx (G_{1:\tilde{K}};\bold{z})$.  Note that we have
\begin{align}
\sum_{k=1}^{\tilde{K}}w_k  \prod_{\ell =1}^L (x_{\ell}-f_{k\ell }) \equiv  
\sum_{k=1}^{\tilde{K}}z_k  \prod_{\ell =1}^L (x_{\ell}-g_{k\ell }) . \label{07}
\end{align}
Without loss of generality, assume that the variable $\ell =L $ is strongly separable in $(F_{1:\tilde{K}};\bold{w})$. Letting   $x_L = f_{\tilde{K}L }$ in (\ref{07}) results
\begin{align}
\sum_{k=1}^{\tilde{K}-1} w_k  (f_{\tilde{K}L}&-f_{kL }) \prod _{\ell =1}^{L-1}(x_{\ell}-f_{k\ell }) \equiv  
 \sum_{k=1}^{\tilde{K}}z_k (f_{\tilde{K}L }-g_{kL })\prod _{\ell =1}^{L-1}(x_{\ell}-g_{k\ell }). \label{08}
\end{align}
Note that the term $k=\tilde{K}$ in LHS of (\ref{07}) becomes zero by choosing $x_L = f_{\tilde{K}L }$. 
Also notice that   for any $k\in [\tilde{K}-1], f_{\tilde{K}L }\neq f_{kL }$, because the variable $\ell =L$ is  strongly separable. 
Now two cases may happen. 

\textit{Case one. } 
First assume that the  RHS of  the summation in (\ref{08}) has less than $\tilde{K}$ non-zero  terms\footnote{
Note that the LHS of (\ref{08}) contains $\tilde{K}-1$ terms, which means that the number of distinct non-zero polynomials  in the summation is $\tilde{K}-1$. This is due to the strong separability of the the variable  $\ell = 1$. 
}, i.e.,  there is a $k \in [\tilde{K}]$, such that $z_k(f_{\tilde{K}L } - g_{kL})=0$. 
Without loss of generality, assume  that $z_{\tilde{K}}(f_{\tilde{K}L} - g_{\tilde{K}L})=0$. 
In this case, we can use the induction hypothesis on the identity in (\ref{08}). 
More precisely,  since  we have set $z_{\tilde{K}}(f_{\tilde{K}L} - g_{\tilde{K}L})=0$, the identity (\ref{08}) can be written as
\begin{align}
\sum_{k=1}^{\tilde{K}-1} \underbrace{
w_k(f_{\tilde{K}L}-f_{kL })}_{w'_k} \prod _{\ell =1}^{L-1}(x_{\ell}-f_{k\ell }) \equiv  
 \sum_{k=1}^{\tilde{K}-1}   \underbrace{
 z_k (f_{\tilde{K}L }-g_{kL })}_{z'_k}\prod _{\ell =1}^{L-1}(x_{\ell}-g_{k\ell })\label{1234}.
\end{align}
We note that $w'_k \neq 0$, since the variable $\ell =1$ has been assumed to be strongly separable and $w_k \neq 0$. Now we consider 
two sets of latent parameters corresponded to the two sides of (\ref{1234}).
Observe that the number of strongly separable variables in the corresponding problem of the LHS of (\ref{1234}) is  at least $2\tilde{K}-2$, which is greater than $2(\tilde{K}-1)-1$. This shows that we can use the induction hypothesis in this case.

Hence,  by using the induction hypothesis, we conclude that  there is a permutation $\pi_1$ on $[\tilde{K}-1]$,  such that for any $k\in [\tilde{K}-1]$ and  $\ell \in [L-1]$, we have $f_{k\ell }=g_{\pi_1(k)\ell }$.  
Now  let  $1 \le\ell_1<\ell_2<\ldots<\ell_{2\tilde{K}-2}\le L-1$ be some  strongly  separable variables of $(F_{1:\tilde{K}};\bold{w})$. The existence of them is guaranteed by the assumption of the induction. Now if we  let   $x_{\ell_k}=f_{k \ell_k  }$ for any $k \in [\tilde{K}-1]$ in (\ref{07}), we have
\begin{align}
& \underbrace{w_{\tilde{K}}  \Big (\prod_{k=1}^{\tilde{K}-1} (f_{k \ell_k }-f_{\tilde{K} \ell_k  }) \Big )}_{w'}  \Big (\prod_{\ell \in [L]\setminus \{\ell_1,\ell_2,\ldots,\ell_{\tilde{K}-1}\}}(x_\ell-f_{\tilde{K}\ell }) \Big )\equiv \nonumber \\
&  \underbrace{ z_{\tilde{K}}  \Big (\prod_{k=1}^{\tilde{K}-1} (f_{k \ell_k  }-g_{\tilde{K} \ell_k  })}_{z'} \Big ) \Big (\prod_{\ell \in [L]\setminus \{\ell_1,\ell_2,\ldots,\ell_{\tilde{K}-1}\}}(x_\ell-g_{\tilde{K}\ell  }) \Big ) .\label{09}
\end{align}
We notice that the above polynomial identity holds due to the fact that each term in the summation in the LHS of (\ref{07}), which corresponds to some $k \in [\tilde{K}-1]$, becomes zero, since we have set $x_{\ell_k}=f_{k \ell_k  }$. Also, for the RHS of (\ref{07}), the  $k^{\text{th}}$ term  in the summation, for any $k\in[\tilde{K}-1]$ becomes zero, due to the fact that we have set $x_{\ell_{k'}} =  f_{k' \ell_{k'}} = g_{k\ell_{k'}}$ in (\ref{07}), where $k' = \pi^{-1}_1(k)$.

Since the variable  $\ell_k$ has been assumed to  be strongly separable in $(F_{1:\tilde{K}};\bold{w})$,  for any $k\in [\tilde{K}-1]$,  we conclude that 
 $w' := w_{\tilde{K}}\prod_{k=1}^{\tilde{K}-1} (f_{k \ell_k }-f_{\tilde{K} \ell_k  }) \neq 0$.
 This shows that (\ref{09})  is a non-zero polynomial. 
Hence by the identity of two (non-zero) polynomials in (\ref{09}), we conclude that 
 $f_{\tilde{K}\ell  }=g_{\tilde{K}\ell  }$ for any $\ell \in [L]\setminus \{\ell_1,\ell_2,\ldots,\ell_{\tilde{K}-1}\}$ and also $w' = z'$.

 A similar argument can be established to show that  we have
 $f_{\tilde{K}\ell  }=g_{\tilde{K}\ell  }$ for any $\ell \in [L]\setminus \{\ell_{\tilde{K}},\ell_{\tilde{K}+1},\ldots,\ell_{2\tilde{K}-2}\}$.
 Combining the results shows that $f_{\tilde{K}\ell  }=g_{\tilde{K}\ell  }$ for any $\ell \in [L]$, since we have
 \begin{align}
 [L] =   ([L]\setminus \{\ell_1,\ell_2,\ldots,\ell_{\tilde{K}-1}\})   \bigcup ( [L]\setminus \{\ell_{\tilde{K}},\ell_{\tilde{K}+1},\ldots,\ell_{2\tilde{K}-2}\}).
\end{align}
 
 By applying this fact to (\ref{09}) and using the identity $w' = z'$, we conclude that that $w_{\tilde{K}}=z_{\tilde{K}}$.
Hence,  the term $k=\tilde{K}$ in the summation in (\ref{07}) can be canceled from two sides. This yields that 
\begin{align}
\sum_{k=1}^{\tilde{K}-1}w_k  \prod_{\ell =1}^L (x_{\ell}-f_{k\ell }) \equiv  
\sum_{k=1}^{\tilde{K}-1}z_k  \prod_{\ell =1}^L (x_{\ell}-g_{k\ell }) . \label{10}
\end{align}
By applying the induction hypothesis to  (\ref{10}), we conclude that  there is a  permutation $\psi$ on $[\tilde{K}-1]$ with the following property
\footnote{
Actually  the permutation $\psi$ is equal to  the permutation $\pi_1$, which is defined previously. However, we do not use this identity and so it is not needed to prove it. 
}. For any $k\in [\tilde{K}-1]$ and $\ell \in [L]$, we have 
$w_k=z_{\psi(k)}$ and $f_{k \ell  }=g_{\psi(k) \ell }$. Combining the results shows that for the permutation $\pi$ on $[\tilde{K}]$, which is defined as 
\[ \pi(k) =
  \begin{cases}
   \psi(k)       & \quad \text{if } k \in [\tilde{K}-1]\\
    \tilde{K}  & \quad \text{if } k=\tilde{K}
  \end{cases}
\]
we have $(F_{1:\tilde{K}};\bold{w})=(G^{\pi}_{1:\tilde{K}};\bold{z}^{\pi})$. This shows that $  (F_{1:\tilde{K}};\bold{w})  \approx (G_{1:\tilde{K}};\bold{z}) $ and completes the proof.

\textit{Case two. }
  Now assume that  the RHS of the summation in (\ref{08}) has exactly $\tilde{K}$ non-zero terms. 
  We notice that the proof of the case one does not depend on the location of the first chosen strongly separable variable.
  In other words, if for some $\ell \in [L]$ and some  $k \in [\tilde{K}]$, where $\ell$ is the location of a strongly separable variable in $(F_{1:\tilde{K}};\bold{w})$, it is possible to set $x_{\ell}=f_{k\ell  }$  in (\ref{07}), such that the assumption in the  case one  holds, then the proof is completed. Hence,  we assume that for any $\ell $, which is the location of a strongly  separable variable  of $(F_{1:\tilde{K}};\bold{w})$, and for any $k\in [\tilde{K}]$, if we set $x_{\ell}=f_{k \ell }$ in (\ref{07}), then the RHS of the result has exactly $\tilde{K}$ non-zero terms, i.e., $f_{k \ell  }\neq g_{k' \ell  }$ for any $k', k \in [\tilde{K}]$ and $z_k \neq 0$ for any $k \in [\tilde{K}]$. Denote the locations of some strongly separable variables in $(F_{1:\tilde{K}};\bold{w})$ by $1 \le\ell_1<\ell_2<\ldots<\ell_{2\tilde{K}-2}\le L-1$. Note that we still assumed that the variable $\ell = L$ is  strongly separable. Without loss of generality, assume that $\ell_{2\tilde{K}-2}=L-1$.  Following by these  assumptions, we set $x_{L-1}=g_{(\tilde{K}-1)(L-1) }$ in (\ref{08}) and conclude that 
\begin{align}
\sum_{k=1}^{\tilde{K}-1} \underbrace{
w_k  a_k}_{w''_k}\prod _{\ell =1}^{L-2}(x_{\ell}-f_{k\ell }) \equiv  \sum_{k \in [\tilde{K}]\setminus \{\tilde{K}-1\}}
\underbrace{z_k b_k}_{z''_k}
\prod _{\ell =1}^{L-2}(x_{\ell}-g_{k\ell }), \label{11} 
\end{align}
where $a_k :=(f_{\tilde{K} L  }-f_{k L  }) (g_{(\tilde{K}-1) (L-1)  }-f_{k(L-1) })$ and $b_k:= (f_{\tilde{K} L  }-g_{k L  }) (g_{(\tilde{K}-1) (L-1)  }-g_{k(L-1) })
 $ for any $k\in [\tilde{K}]$.  Note that because of the discussed considerations, two sides of the summation in (\ref{11}) have exactly $\tilde{K}-1$ terms in two sides, i.e.,  $a_k \neq 0$ for any $k\in [\tilde{K}-1]$ and $b_k \neq 0$ for any $k \in [\tilde{K}] \setminus \{\tilde{K}-1\}$. 
 Also, the number of the residual strongly separable variables of the corresponding problem of the LHS of (\ref{11}) is at least $2\tilde{K}-3$. Hence,  using the induction hypothesis, we conclude that there is a bijection mapping $\phi: [\tilde{K}-1]\rightarrow [\tilde{K}]\setminus \{\tilde{K}-1\}$  such that for any $\ell \in [L-2]$ and $k \in [\tilde{K}-1]$
 we have $f_{k\ell  }=g_{ \phi (k) \ell  }$. Observe that because $\tilde{K}\ge 2$, we have $\ell_1 \in [L-2]$. This shows that $f_{k\ell_1  }=g_{ \phi (k) \ell_1  }$ for any $k \in [\tilde{K}-1]$. Note that this is  contradiction with the assumption that we make for the case two, where now if we set $x_{\ell_1} = f_{1 \ell_1 }$ in (\ref{07}), then the RHS of the result has less than $\tilde{K}$ terms. This shows that  it is impossible that the first case does not happen. This completes the proof.

\end{proof}

\subsection{Proof of the sufficiency part of Theorem 1 for $M\ge 3$} 
In this part, we focus on  larger state spaces  than  binaries. The main idea for the proof is that by introducing some auxiliary binary mixture models, we can use the result of Theorem \ref{Theorem3}. The auxiliary binary mixture models are generated based on the projection of the variables into binary spaces. This allows us to use Theorem \ref{Theorem3}.

Consider latent variables  $(F_{1:K};\bold{w}) \in \Theta_{K,L,M}$, where $\mathcal{L}_s(F_{1:K};\bold{w})\ge 2K-1$ and $M\ge 3$.  Also assume that   $(G_{1:K};\bold{z}) \in \overline{\Theta}_{K,L,M}$  and $\bold{f}=\bold{g}$.  Note that $\bold{f},\bold{g}$ are  the mixture distributions of the problems. We  want to show that $(G_{1:K};\bold{z}) \approx (F_{1:K};\bold{w})$. 
 We introduce two auxiliary binary mixture models $(\tilde{F}_{1:K};\bold{w})\in \Theta_{K,L,2}$ and $(\tilde{G}_{1:K};\bold{z})\in \overline{\Theta}_{K,L,2}$  as follows. For any $k\in [K]$ and $\ell \in [L]$, let $\tilde{f}_{k\ell 1}=f_{k \ell 1}$ and $\tilde{g}_{k\ell 1}=g_{k \ell 1}$. Also  let  $\tilde{f}_{k \ell 2} = 1-f_{k \ell 1}$ and  $\tilde{g}_{k \ell 2} = 1-g_{k \ell 1}$.  Note that by the definition of the strongly separable variables, they do not waste by this transformation. 
 In other words, if the $\ell^{\text{th}}$ variable is strongly separable in $({F}_{1:K};\bold{w})$, then it is strongly separable in $(\tilde{F}_{1:K};\bold{w})$, too. 
 Hence,  we have $\mathcal{L}_s(\tilde{F}_{1:K};{\bold{w}})\ge\mathcal{L}_s(F_{1:K};\bold{w})\ge 2K-1 $.
  
 \begin{lemma}\label{lemma2}  $C_{(\tilde{F}_{1:K};\bold{w})}(x_{1:L}) \equiv C_{(\tilde{G}_{1:K};\bold{z})}(x_{1:L})$.
 \end{lemma}
 \begin{proof}
 See appendix B.
 \end{proof}
Using Lemma \ref{lemma2} and Theorem \ref{Theorem3}, we conclude that $(\tilde{F}_{1:K};{\bold{w}})\approx (\tilde{G}_{1:K};{\bold{z}})$.  This implies that there is a permutation $\pi$ on $[K]$ such that $(\tilde{F}_{1:K};{\bold{w}})=(\tilde{G}^{\pi}_{1:K};{\bold{z}}^{\pi})$.  This shows that for any $\ell \in [L]$ and any $k\in [K]$ we have $\tilde{f}_{k\ell 1}=\tilde{g}_{\pi(k)\ell 1} $. Using the definitions of the auxiliary problems,  we conclude that $f_{k\ell 1}=g_{\pi(k)\ell 1} $ for any $\ell \in [L]$ and $k\in [K]$.  It is also concluded that $w_k=z_{\pi(k)}$ for any $k\in[K]$.

 Now  we claim that $(F_{1:K};\bold{w})=(G^{\pi}_{1:K};\bold{z}^{\pi})$ and this   completes the proof. For this purpose, we need to show that for any $\ell \in [L]$, $k\in[K]$ and $m\in[M]\setminus \{1\}$ we have $f_{k\ell m}=g_{\pi(k)\ell m} $. Using the symmetry in  the problem, it suffices to prove that $f_{112}=g_{\pi(1)12}$.
 
  Again, we introduce  two  auxiliary binary mixture models $(F'_{1:K};\bold{w})\in \Theta_{K,L,2}$ and $(G'_{1:K};\bold{z})\in \overline{\Theta}_{K,L,2}$  as follows.   For any $k\in [K]$ and $\ell \in [L]\setminus \{1\}$, let ${f}'_{k\ell 1}=f_{k \ell 1}$ and   ${g}'_{k\ell 1}=g_{k \ell 1}$. For any $k\in[K]$,  let  ${f}'_{k1 1}=f_{k 12}$ and  ${g}'_{k1 1}=g_{k 12}$. Also, let  ${f'}_{k\ell 2} = 1-{f'}_{k\ell 1}$ and  ${g'}_{k\ell 2} = 1-{g'}_{k\ell 1}$ for any $k \in [K]$ and $\ell \in [L]$. Similarly,  we have $\mathcal{L}_s({F'}_{1:K};\bold{w})\ge\mathcal{L}_s({F}_{1:K};\bold{w})\ge 2K-1 $.

 \begin{lemma}\label{lemma3}  $C_{({F'}_{1:K};\bold{w})}(x_{1:L}) \equiv C_{({G'}_{1:K};\bold{z})}(x_{1:L})$.
 \end{lemma}
\begin{proof}
See appendix C.
\end{proof}

Using Lemma \ref{lemma3} and Theorem \ref{Theorem3}, we conclude that $({F'}_{1:K};{\bold{w}})\approx ({G'}_{1:K};{\bold{z}})$. 
This yields that there is a permutation $\psi$ on $[K]$,  such that we have
  $({F'}_{1:K};\bold{w})=({G'}_{1:K}^{\psi};\bold{z}^{\psi})$. Hence,  for any  $\ell \in [L]$ and any $k\in [K]$,  we have ${f'}_{k\ell 1}={g'}_{\psi(k)\ell 1}$. Thus,  we conclude that $f_{k\ell 1}=g_{\psi(k)\ell 1}$ for any $\ell \in  [L]\setminus \{1\} $ and $k\in [K]$.  Choose $\ell \in [L]\setminus \{1\}$ such that the $\ell^{\text{th}}$ variable is strongly separable in  $({F}_{1:K};\bold{w})$. This is possible due to  the assumption of the  existence of at least $2K-1$ strongly separable variables
\footnote{For the case $K=1$, this assumption may be incorrect. However, the proof for the case $K=1$ is trivial and can be done directly.}
 in $({F}_{1:K};\bold{w})$. We claim that $\psi = \pi$. Note that for any $k\in[K]$ we have $f_{k\ell 1}=g_{\pi(k)\ell 1} = f_{\psi^{-1}(\pi(k)) \ell 1}$. 
 Note that because of the strong  separability of the $\ell^{\text{th}}$ variable,
 the set $\{f_{k\ell 1} \in [0,1] | k \in [K]\}$ has exactly $K$ elements. This shows that $\psi^{-1}(\pi(k))=k$ or $\psi(k)=\pi(k)$ for any $k \in [K]$. Hence we have $\psi=\pi$. 
 
 Now using $({F'}_{1:K};\bold{w})= ({G'}^{\psi}_{1:K};\bold{z}^{\psi})= ({G'}^{\pi}_{1:K};\bold{z}^{\pi})$,  we conclude that ${f'}_{111}={g'}_{\pi(1)11}$. It is also assumed that  ${f'}_{111}=f_{112}$ and ${g'}_{\pi(1)11}=g_{\pi(1)12}$. Hence,  we have $f_{112}=g_{\pi(1)12}$. Similarly,  we can conclude that $(F_{1:K};\bold{w})=(G_{1:K}^{\pi};\bold{z}^{\pi})$ and  hence, the proof is completed.

\subsection{Proof of the necessary part of Theorem 1}

In this part, for any $\overline{L} \in [\min(2K-2,L)]\cup \{0\}$ we introduce  a problem $(F_{1:K};\bold{w}) \in \Theta_{K,L,M}$ such that  $\mathcal{L}_s(F_{1:K};\bold{w}) =  \overline{L}$ and $(F_{1:K};\bold{w}) $  is not identifiable. In particular, we introduce two problems $(F_{1:K};\bold{w}) ,(G_{1:K};\bold{z})\in \Theta_{K,L,M}$, where $\mathcal{L}_s(F_{1:K};\bold{w}) = \mathcal{L}_s(G_{1:K};\bold{z}) =  \overline{L}$, such that $\bold{f} = \bold{g}$ (they have the  same mixture distributions) and $(F_{1:K};\bold{w})  \not \approx (G_{1:K};\bold{z})$.

For any positive constants $\alpha$ and $\beta$, we introduce two problems $(F_{1:K};\bold{w}) ,(G_{1:K};\bold{z})\in \Theta_{K,L,M}$ as follows\footnote{The case $\overline{L} = 0$ is trivial. Note that if two mixture components   have the same frequencies, then $\overline{L}=0$ and $(F_{1:K};\bold{w}) $ is not identifiable, where the probability of  sampling from each of that two mixture components can not be determined from the mixture distribution. Hence, we assume that $\overline{L}\ge 1.$
}. First we set 
$w_k =  {  2K- 1 \choose  2k-2}/{2^{2K-1}}$ and $z_k =  {  2K- 1 \choose  2k-1}/{2^{2K-1}}$ for any $k \in [K]$. Note that $\sum_{k=1}^K w_k=\sum_{k=1}^K z_k = 1$. Then, for any $k\in [K]$, $\ell \in [\overline{L}]$ and $m \in [M-1]$, we define $f_{k\ell m} = \alpha (2k-2) + \beta m $ and  $g_{k\ell m} = \alpha (2k-1) + \beta m $. Also for any $k\in [K]$, $\ell \in [L] \setminus  [\overline{L}]$ and $m \in [M-1]$, we set $f_{k\ell m}=g_{k \ell m} = 1/M$. The values of $f_{k\ell M}$ and $g_{k\ell M}$ are determined by the condition   $\sum_{m \in [M]}f_{k\ell m}=\sum_{m \in [M]}g_{k\ell m}=1.$ Note that  in this definition, $(F_{1:K};\bold{w}) ,(G_{1:K};\bold{z})\in \Theta_{K,L,M}$ if we choose $\alpha$ and $\beta$ positive and small enough. 
Now we claim that these two problems work for the proof of the necessary part of Theorem \ref{Theorem1}.

 First we notice that  $\mathcal{L}_s(F_{1:K};\bold{w}) = \mathcal{L}_s(G_{1:K};\bold{z}) =  \overline{L}$. 
 Second, it is obvious that $(F_{1:K};\bold{w})  \not \approx (G_{1:K};\bold{z})$, because the frequencies of two problems in the first $\ell \in [\overline{L}]$ variables are essentially different from each other. 
 Hence, for completing the proof, it remains to show that we have $\bold{f} = \bold{g}$. In particular, we are interested to show that for any $\bold{m}=(m_1,m_2\ldots,m_L) \in [M]^L$, we have $\bold{f}{(m_1,m_2\ldots,m_L)}  = \bold{g}{(m_1,m_2\ldots,m_L)},$ or  equivalently,  $\sum_{k=1}^K w_k \prod_{\ell=1}^{L} f_{k \ell m_{\ell}}= \sum_{k=1}^K z_k \prod_{\ell=1}^{L} g_{k \ell m_{\ell}}$. First we state the following lemma about the defined problems $(F_{1:K};\bold{w})$ and  $(G_{1:K};\bold{z})$.

\begin{lemma}\label{lemma4}
If  for any   $I  \subseteq [\overline{L}]$ and $\bold{m}=(m_\ell)_{\ell \in I} \in [M-1]^{|I|}$ we have
\footnote{If $I = \emptyset$, we define $\prod_{\ell \in  I} f_{k \ell m_{\ell}}=1.$}
 $\sum_{k=1}^K w_k \prod_{\ell \in  I} f_{k \ell m_{\ell}}= \sum_{k=1}^K z_k \prod_{\ell \in I}  g_{k \ell m_{\ell}}$,  then  $\bold{f} = \bold{g}$. 
\end{lemma}

\begin{proof}
See appendix D.
\end{proof}

Using Lemma \ref{lemma4}, for the proof of  $\bold{f} = \bold{g}$, it suffices to prove that  for any   $I \subseteq [\overline{L}]$ and $\bold{m}=(m_\ell)_{\ell \in I} \in [M-1]^{|I|}$, the identity $\sum_{k=1}^K w_k \prod_{\ell  \in  I}  f_{k \ell m_{\ell}}= \sum_{k=1}^K z_k \prod_{\ell \in  I}  g_{k \ell m_{\ell}}$ holds. Observe that

\begin{align}
  \sum_{k=1}^K w_k \prod_{\ell \in I} f_{k \ell m_{\ell}} -  \sum_{k=1}^K z_k \prod_{\ell  \in  I }   g_{k \ell m_{\ell}} & = 
    \sum_{k=1}^K   \frac{{  2K- 1 \choose  2k-2}}{2^{2K-1}}\prod_{\ell  \in  I} (\alpha (2k-2) +\beta m_{\ell} )   \\
  &   - 
  \sum_{k=1}^K  \frac{{  2K- 1 \choose  2k-1}}{2^{2K-1}}  \prod_{\ell  \in I}  (\alpha (2k-1) +\beta m_{\ell} )    \nonumber \\
 & = \sum_{i=0}^{2K-1}   \frac{{  2K- 1 \choose  i}}{2^{2K-1}} (-1)^i \prod_{\ell  \in  I}  (\alpha i +\beta m_{\ell} ) \label{100}.
\end{align} 
We want to show that (\ref{100}) is equal to zero. Let us define an $M-1$ variable real function $h(x_{1:M-1}) $ as 
\begin{align}
h(x_{1:M-1}) := \exp(\beta \times  \sum_{m=1}^{M-1} m x_m)(1- \exp(\alpha \times \sum_{m=1}^{M-1} x_m))^{2K-1}.
\end{align}

\begin{lemma}\label{lemma5}
For any $ \bold{t} = (t_1,t_2,\ldots,t_{M-1})^T   \in \{0,1,\ldots,2K-2\}^{M-1} $, such that $t = \sum_{m=1}^{M-1} t_m \le 2K-2$,  we have 
\begin{align}
\frac{\partial^t  h}{\partial x^{t_1}_1  \partial x^{t_2}_2 \cdots  \partial x^{t_{M-1}}_{M-1}} (0,0,\ldots, 0) = 0. 
\end{align}
\end{lemma}

\begin{proof}
See appendix E.
\end{proof}

Using Lemma \ref{lemma5}, we aim to prove that (\ref{100}) is equal to zero. Note that if we define  $t_m := |\{ \ell \in I : m_\ell = m\}| \in [2K-2] \cup \{0\}$ for any $m \in [M-1]$ and $t := \sum_{m=1}^{M-1} t_m  =  |I| \le \overline{L} \le 2K-2$, then we have 
\begin{align}
\sum_{i=0}^{2K-1}   \frac{{  2K- 1 \choose  i}}{2^{2K-1}} (-1)^i \prod_{\ell  \in  I}  (\alpha i +\beta m_{\ell} ) &=\sum_{i=0}^{2K-1}   \frac{{  2K- 1 \choose  i}}{2^{2K-1}} (-1)^i \prod_{m=1}^{M-1}(\alpha i +\beta m )^{t_m}\\ &\overset{(a)}{=}  \frac{\partial^t  h}{\partial x^{t_1}_1  \partial x^{t_2}_2 \cdots  \partial x^{t_{M-1}}_{M-1}} (0,0,\ldots, 0)  \label{101},
\end{align} 
which is equal to zero based on Lemma \ref{lemma5}. Note that (a) follows from  the expansion of the function $h(x_{1:M-1})$ as follows
\begin{align*}
h(x_{1:M-1}) &= \exp(\beta \times  \sum_{m=1}^{M-1} m x_m)(1- \exp(\alpha \times \sum_{m=1}^{M-1} x_m))^{2K-1} \\&= \sum_{i=0}^{2K-1}\frac{{  2K- 1 \choose  i}}{2^{2K-1}} (-1)^i \exp((\alpha i + \beta m )  \times (\sum_{m=1}^{M-1}  x_m)).
\end{align*}
We are done. 

\section{Proof of Theorem 2}
\label{Proof of Theorem 2}
In this section, we prove the sufficiency of $2K$ weakly separable variables for the identifiability. 
First we note that for the binary state spaces two notions of weakly separable variables and strongly separable variables are equivalent.
Therefore,  according to the proof of Theorem \ref{Theorem1}, we conclude the desired result for    $M=2$. 
Hence we consider the cases that $M\ge 3$.  

We establish a proof based on introducing the auxiliary binary mixture models, which is very similar to the proof of Theorem \ref{Theorem1} for the case $M\ge 3$.  
Consider latent variables  $(F_{1:K};\bold{w}) \in \Theta_{K,L,M}$, where $\mathcal{L}_w(F_{1:K};\bold{w})\ge 2K-1$ and $M\ge 3$.  Also assume that   $(G_{1:K};\bold{z}) \in \overline{\Theta}_{K,L,M}$  and $\bold{f}=\bold{g}$.
This means that the two problems have the same mixture distributions. We  aim to prove that $(G_{1:K};\bold{z}) \approx (F_{1:K};\bold{w})$. 

Without loss of generality, assume that the variables 
$
1 \le \ell_1 <\ell_2<\ldots<\ell_{2K}\le L
$
are weakly separable in  $(F_{1:K};\bold{w})$. 
This means that for any $r \in [2K]$, there is an $m_{r} \in [M]$ such that $f_{k \ell_r m_r} \neq f_{k' \ell_r m_r}$ for any distinct $k,k' \in [K]$. Without loss of generality, 
assume that $m_r = 1$ for any $r \in [2K]$. 
Let us define the set 
$I = \{ \ell_r | r \in [2K] \}$.

Similar to the proof of Theorem \ref{Theorem1} for $M \ge 3$, we introduce two auxiliary binary mixture models $(\tilde{F}_{1:K};\bold{w})\in \Theta_{K,L,2}$ and $(\tilde{G}_{1:K};\bold{z})\in \overline{\Theta}_{K,L,2}$  as follows.
 For any $k\in [K]$ and $\ell \in [L]$, let $\tilde{f}_{k\ell 1}=f_{k \ell 1}$ and $\tilde{g}_{k\ell 1}=g_{k \ell 1}$. Also  let  $\tilde{f}_{k \ell 2} = 1-f_{k \ell 1}$ and  $\tilde{g}_{k \ell 2} = 1-g_{k \ell 1}$.  Note that by the definition of the weakly separable variables, all of the variables $\ell \in I$ are strongly separable in $(\tilde{F}_{1:K};\bold{w})$. 
This means that $\mathcal{L}_s(\tilde{F}_{1:K};{\bold{w}})\ge2K $. 

Using Lemma \ref{lemma2} and Theorem \ref{Theorem1}, we conclude that $(\tilde{F}_{1:K};{\bold{w}})\approx (\tilde{G}_{1:K};{\bold{z}})$.  This means that there is a permutation $\pi$ on $[K]$ such that $(\tilde{F}_{1:K};{\bold{w}})=(\tilde{G}^{\pi}_{1:K};{\bold{z}}^{\pi})$.  This shows that for any $\ell \in [L]$ and any $k\in [K]$ we have $\tilde{f}_{k\ell 1}=\tilde{g}_{\pi(k)\ell 1} $. Using the definitions of the auxiliary problems,  we conclude that $f_{k\ell 1}=g_{\pi(k)\ell 1} $ for any $\ell \in [L]$ and $k\in [K]$.  It is also concluded that $w_k=z_{\pi(k)}$ for any $k\in[K]$.

 Now  we claim that $(F_{1:K};\bold{w})=(G^{\pi}_{1:K};\bold{z}^{\pi})$ and this   completes the proof. For this purpose, we need to show that for any $\ell \in [L]$, $k\in[K]$ and $m\in[M]\setminus \{1\}$ we have $f_{k\ell m}=g_{\pi(k)\ell m} $. Using the symmetry in  the problem, it suffices to prove that $f_{112}=g_{\pi(1)12}$.
 
 Again, we introduce  two  auxiliary binary mixture models $(F'_{1:K};\bold{w})\in \Theta_{K,L,2}$ and $(G'_{1:K};\bold{z})\in \overline{\Theta}_{K,L,2}$  as follows.   For any $k\in [K]$ and $\ell \in [L]\setminus \{1\}$, let ${f}'_{k\ell 1}=f_{k \ell 1}$ and   ${g}'_{k\ell 1}=g_{k \ell 1}$. For any $k\in[K]$,  let  ${f}'_{k1 1}=f_{k 12}$ and  ${g}'_{k1 1}=g_{k 12}$. Also, let  ${f'}_{k\ell 2} = 1-{f'}_{k\ell 1}$ and  ${g'}_{k\ell 2} = 1-{g'}_{k\ell 1}$ for any $k \in [K]$ and $\ell \in [L]$. 
 
 Now two cases may occur. First assume that $1 \not \in I$. In this case, we have the inequality $\mathcal{L}_s({F'}_{1:K};\bold{w})\ge 2K$, due to the weak separability of the variables of the set $I$. 
 
 For the second case, assume that $1 \in I$. This shows that when we project the first variable into the specific binary space, which is defined, the first weakly separable variable may waste. However, the other weakly separable variables of  $({F}_{1:K};\bold{w})$ hold in $({F'}_{1:K};\bold{w})$. Thus, we conclude that $\mathcal{L}_s({F'}_{1:K};\bold{w})\ge 2K-1$.
 
Hence, for the two cases we conclude that $\mathcal{L}_s({F'}_{1:K};\bold{w})\ge 2K-1$. The rest of the proof is exactly similar to the sufficiency proof of Theorem \ref{Theorem1} for the case $M \ge 3$ and so it is omitted. We are done.

\begin{remark} \normalfont
In the worst-case regime, our counterexample for the identifiability of mixture models with less than $2K-1$ strongly separable variables is also valid for studying the  weakly separable variables. In other words, there are mixture models with less than $2K-1$ weakly separable variables which are not identifiable. Hence, the optimal threshold for the weakly separable variables is achieved in this paper  within at most one variable. 
\end{remark}

\appendix

\section{Proof of Lemma 1}
In order to prove, it suffices to show that for any $(F_{1:K};\bold{w}),(G_{1:K};\bold{z})\in \overline{\Theta}_{K,L}$ with mixture distributions  $\bold{f}$ and $\bold{g}$,   the identity 
 $\bold{f}=\bold{g}$ implies $C_{(F_{1:K};\bold{w})}(x_{1:L})  \equiv C_{(G_{1:K};\bold{z})}(x_{1:L})$ and vice versa.  
 This is due to the definitions of the paper. 

  First assume that  we have the identity 
$C_{(F_{1:K};\bold{w})}(x_{1:L}) \equiv  C_{(G_{1:K};\bold{z})}(x_{1:L}).$
 Note that 
\begin{align}
C_{(F_{1:K};\bold{w})}(x_{1:L})&=\sum_{k=1}^K w_k   \prod_{\ell =1}^L(x_\ell-f_{k\ell})  \nonumber \\
&=\sum_{k=1}^K w_k  \Big \{ \sum_{I\subseteq [L]}  (-1)^{|I|} \Big (\prod_{\ell \in  [L]\setminus I} x_\ell \Big ) \Big (  \prod_{\ell \in I} f_{k \ell } \Big )  \Big \}    \nonumber \\
& = \sum_{I\subseteq [L]}  (-1)^{|I|} \Big (\prod_{\ell \in  [L]\setminus I} x_\ell  \Big )  \Big (   \sum_{k=1}^K w_k    \prod_{\ell \in  I} f_{k \ell }  \Big )  \label{5}.
\end{align}
This implies that for any  $I \subseteq [L]$ we have 
\begin{align}
\sum_{k=1}^K w_k    \prod_{\ell \in  I} f_{k \ell }=\sum_{k=1}^K z_k \prod_{\ell \in  I} g_{k\ell }  \label{03}.
\end{align}
For any $\bold{m}=(m_1,m_2,\ldots,m_L)\in \{1,2\}^L$, define $I_{\bold{m}}:=\{\ell \in [L]: m_\ell = 1\}$. Observe that 
\begin{align}
\bold{f}{(m_1,m_2,\ldots,m_L)} &=  \sum_{k=1}^K w_k  \Big( \prod_{\ell=1}^L   f_{k \ell  m_\ell} \Big) \nonumber \\
&= \sum_{k=1}^K  w_k  \Big( \prod_{\ell \in I_{\bold{m}}}f_{k \ell  } \prod_{\ell \in [L]\setminus I_{\bold{m}}} (1-f_{k \ell  }) \Big)   \nonumber \\
&= \sum_{k=1}^K w_k  \Big( \sum_{I_{\bold{m}}\subseteq I\subseteq [L]}(-1)^{|I|-|I_{\bold{m}}|}\prod_{\ell \in I} f_{k \ell  } \Big)  \nonumber  \\
&=  \sum_{I_{\bold{m}}\subseteq I\subseteq [L]}(-1)^{|I|-|I_{\bold{m}}|}\Big ( \sum_{k=1}^K w_k   \prod_{\ell \in I} f_{k \ell  } \Big ) \label{04}.
\end{align}
Using (\ref{03}) and (\ref{04}),  we conclude that $\bold{f}{(m_1,m_2,\ldots,m_L)}=\bold{g}{(m_1,m_2,\ldots,m_L)}$  for any $\bold{m}=(m_1,m_2,\ldots,m_L)\in \{1,2\}^L$. This shows that $\bold{f}=\bold{g}$  and  concludes the desired result.

 Now for the other side,  assume that $\bold{f}=\bold{g}$. We will show that $C_{(F_{1:K};\bold{w})}(x_{1:L}) \equiv C_{(G_{1:K};\bold{z})}(x_{1:L})$.  Using the identity in  (\ref{5}),  it suffices to show that for any $I\subseteq [L]$ we have
\begin{align}
\sum_{k=1}^K w_k    \prod_{\ell \in  I} f_{k\ell  }=\sum_{k=1}^K  z_k \prod_{\ell \in  I} g_{k \ell  }. \label{06}
\end{align}
Note that for any $I\subseteq [L]$ we  have
\begin{align}
\sum_{k=1}^K w_k   \prod_{\ell \in  I} f_{k \ell  }  &= \sum_{k=1}^K  w_k   \Big( \sum_{\bold{m} \in \{1,2\}^L: I\subseteq I_\bold{m}}\prod_{\ell=1}^L f_{k \ell m_\ell} \Big) \nonumber \\
&= \sum_{\bold{m} \in \{1,2\}^L: I\subseteq I_\bold{m}} \Big(  \sum_{k=1}^K w_k  \prod_{\ell=1}^L  f_{k \ell m_\ell}   \Big) \nonumber \\
& = \sum_{\bold{m} \in \{1,2\}^L: I\subseteq I_\bold{m}} \bold{f}{(m_1,m_2,\ldots,m_L)}.
\end{align}
This shows that if $\bold{f}=\bold{g}$, then  for any $I\subseteq [L]$ we can conclude (\ref{06}). Hence, we have $C_{(F_{1:K};\bold{w})}(x_{1:L}) \equiv C_{(G_{1:K};\bold{z})}( x_{1:L})$ and this completes the proof.

\section{Proof of Lemma 2}
Using (\ref{5}),  it suffices  to prove that   for any $I\subseteq [L]$ we have
\begin{align}
\sum_{k=1}^K w_k    \prod_{\ell \in  I} \tilde{f}_{k\ell 1}=\sum_{k=1}^K  z_k \prod_{\ell \in  I} \tilde{g}_{k \ell 1}\label{112}.
\end{align}
For any $\bold{m}=(m_1,m_2,\ldots,m_L)\in [M]^L$, define $I_{\bold{m}}:=\{\ell \in [L]: m_\ell = 1\}$.  Note that for any $I\subseteq [L]$ we  have
\begin{align*}
\sum_{k=1}^K w_k   \prod_{\ell \in  I} f_{k \ell 1}  &= \sum_{k=1}^K  w_k   \Big( \sum_{\bold{m} \in  [M]^L: I\subseteq I_\bold{m}}\prod_{\ell=1}^L f_{k \ell m_\ell} \Big) \nonumber \\
&= \sum_{
\substack{
\bold{m} \in [M]^L\\
I\subseteq I_\bold{m}
}
} \Big(  \sum_{k=1}^K w_k  \prod_{\ell=1}^L  f_{k \ell m_\ell}   \Big) \nonumber \\
& = \sum_{
\substack{
\bold{m} \in [M]^L\\
 I\subseteq I_\bold{m}}
 } \bold{f}{(m_1,m_2,\ldots,m_L)}.
\end{align*}
Hence, using the assumption $\bold{f}=\bold{g}$, for any $I \subseteq [L]$ we  have  $\sum_{k=1}^K w_k    \prod_{\ell \in  I} {f}_{k\ell 1}=\sum_{k=1}^K  z_k \prod_{\ell \in  I} {g}_{k \ell 1}$, which implies (\ref{112}). Thus, the proof is completed.

\section{Proof of Lemma 3}
Similar to the proof of lemma 2, it suffices to show that for any $I\subseteq [L]$ we have
\begin{align}
\sum_{k=1}^K w_k    \prod_{\ell \in  I} {f'}_{k\ell 1}=\sum_{k=1}^K  z_k \prod_{\ell \in  I} {g'}_{k \ell 1}\label{113}.
\end{align}
If $1 \notin I$, similar to the proof of Lemma 2, we conclude (\ref{113}). 
For any $\bold{m}=(m_1,m_2,\ldots,m_L)\in \{1,2\}^L$, define $I_{\bold{m}}:=\{\ell \in [L]: m_\ell = 1\}$. Assume that $1 \in I$. We have 
\begin{align*}
\sum_{k=1}^K w_k   \prod_{\ell \in  I} f'_{k \ell 1}  &= \sum_{k=1}^K  w_k   \Big( \sum_{
\substack{
\bold{m} \in  [M]^L\\
 I \setminus \{1 \}\subseteq I_\bold{m}\\
  m_1 = 2}
  }\prod_{\ell=1}^L f'_{k \ell m_\ell} \Big) \nonumber \\
&= \sum_{
\substack{
\bold{m} \in  [M]^L\\
 I \setminus \{1 \}\subseteq I_\bold{m}\\
  m_1 = 2}
} \Big(  \sum_{k=1}^K w_k  \prod_{\ell=1}^L  f'_{k \ell m_\ell}   \Big) \nonumber \\
& = \sum_{
\substack{
\bold{m} \in  [M]^L\\
 I \setminus \{1 \}\subseteq I_\bold{m}\\
  m_1 = 2}
} \bold{f}_{(m_1,m_2,\ldots,m_L)}.
\end{align*}
Hence, using the assumption $\bold{f}=\bold{g}$, for any $I \subseteq [L]$ we  have  $\sum_{k=1}^K w_k    \prod_{\ell \in  I} {f'}_{k\ell 1}=\sum_{k=1}^K  z_k \prod_{\ell \in  I} {g'}_{k \ell 1}$. This completes the proof.

\section{Proof of Lemma 4}
For the proof of the lemma, it suffices to show that for any $\bold{m}=(m_1,m_2\ldots,m_{L}) \in [M]^L$, we have $\bold{f}{(m_1,m_2\ldots,m_{L}) }=\bold{g}{(m_1,m_2\ldots,m_{L}) }$. Let us define $B_{\bold{m}} := \{\ell \in [\overline{L}] : m_\ell \neq M\}$. We write
\begin{align*}
\bold{f}{(m_1,m_2\ldots,m_{L}) } &=  \sum_{k=1}^K w_k \prod_{\ell=1}^{L} f_{k \ell m_{\ell}} \\
&= \sum_{k=1}^K w_k \prod_{\ell=1}^{\overline{L}} f_{k \ell m_{\ell}}\times (\frac{1}{M^{L-\overline{L}}})  \\
&= \sum_{k=1}^K w_k   (\frac{1}{M^{L-\overline{L}}}) \Big \{\prod_{\ell \in  B_{\bold{m}}} f_{k \ell m_{\ell}} \times \prod_{\ell \in [\overline{L}] \setminus B_{\bold{m}}}(1-\sum_{m\in [M-1]}{f_{k\ell m}}) \Big \} \\ 
&= (\frac{1}{M^{L-\overline{L}}}) \sum_{k=1}^K w_k  \Big \{    \sum_{ 
\substack{
           B_{\bold{m}} \subseteq I \subseteq [\overline{L}]\\
           (n_\ell)_{\ell \in I} \in [M-1]^{|I|}\\
           \forall \ell \in B_{\bold{m}}:  n_\ell = m_\ell
           }
}  
  (-1)^{|I| - |B_{\bold{m}}|}  \prod_{\ell \in I} f_{k \ell n_\ell}
\Big\}  \\
&= (\frac{1}{M^{L-\overline{L}}})   \sum_{ \substack{
           B_{\bold{m}} \subseteq I \subseteq [\overline{L}]\\
           (n_\ell)_{\ell \in I} \in [M-1]^{|I|}\\
           \forall \ell \in B_{\bold{m}}:  n_\ell = m_\ell
           }
}    (-1)^{|I| - |B_{\bold{m}}|}   \sum_{k=1}^K w_k  \prod_{\ell \in I} f_{k \ell n_\ell}.
\end{align*}
Hence, if  $ \sum_{k=1}^K w_k  \prod_{\ell \in I} f_{k \ell n_\ell} =  \sum_{k=1}^K z_k  \prod_{\ell \in I} g_{k \ell n_\ell}$ for any $I \subseteq [\overline{L}]$ and $(n_\ell)_{\ell \in I} \in [M-1] ^{|I|}$, then $\bold{f}{(m_1,m_2\ldots,m_{L}) }=\bold{g}{(m_1,m_2\ldots,m_{L}) }$ for any $\bold{m} \in [M]^L$. This completes the proof.

\section{Proof of Lemma 5}
First define  the two functions 
\begin{align*}
f(x_{1:M-1}) := \exp(\beta \times  \sum_{m=1}^{M-1} m x_m)
\end{align*}
and
\begin{align*}
 g(x_{1:M-1})  := (1- \exp(\alpha \times \sum_{m=1}^{M-1} x_m))^{2K-1}.
\end{align*} 
Observe that $h(x_{1:M-1})  = f(x_{1:M-1})  \times g(x_{1:M-1}) $. Now we write
\begin{align*}
&\frac{\partial^t  h}{\partial x^{t_1}_1  \partial x^{t_2}_2 \cdots  \partial x^{t_{M-1}}_{M-1}} (0,0,\ldots, 0)  = \\ & \sum_{
\substack{
            \bold{u},\bold{v} \in \{0,1,\ldots,2K-2\}^{M-1}\\
          \bold{u}+\bold{v} = \bold{t}}
}
 C(\bold{u},\bold{v}) \times \frac{\partial^u  f}{\partial x^{u_1}_1  \partial x^{u_2}_2 \cdots  \partial x^{u_{M-1}}_{M-1}} (0,0,\ldots, 0) 
\times 
\frac{\partial^v  g}{\partial x^{v_1}_1  \partial x^{v_2}_2 \cdots  \partial x^{v_{M-1}}_{M-1}} (0,0,\ldots, 0),
\end{align*}
where $u = \sum_{m=1}^{M-1} u_m$ and $v = \sum_{m=1}^{M-1} v_m$. Also $C(\bold{u},\bold{v}) \in \mathbb{N}$ is a positive integer constant, which is independent of $x_m$s.

For completing the proof, it suffices to show that 
\begin{align*}
\frac{\partial^v  g}{\partial x^{v_1}_1  \partial x^{v_2}_2 \cdots  \partial x^{v_{M-1}}_{M-1}} (0,0,\ldots, 0) = 0,
\end{align*} 
for any $\bold{v} \in \{0,1,\ldots,2K-2\}^{M-1}$ such that $v = \sum_{m=1}^{M-1} v_m \le 2K-2$.
  Note that   we have
  \begin{align*}
\frac{\partial^v  g}{\partial x^{v_1}_1  \partial x^{v_2}_2 \cdots  \partial x^{v_{M-1}}_{M-1}}  = C(\bold{v}) \times 
(1- \exp(\alpha \times \sum_{m=1}^{M-1} x_m))^{2K-1 - v} ,
\end{align*} 
 for some  constant $C(\bold{v})$. We conclude the desired result as we have $v \le 2K-2< 2K-1$.

\end{document}